\documentclass[11pt]{amsart}

\usepackage{enumerate,url,amssymb,  mathrsfs, nicefrac, pifont, stmaryrd, graphicx, pdfsync}
\newtheorem{theorem}{Theorem}[section]
\newtheorem{lemma}[theorem]{Lemma}
\newtheorem*{lemma*}{Lemma}

\theoremstyle{definition}
\newtheorem{definition}[theorem]{Definition}

\theoremstyle{remark}
\newtheorem{remark}[theorem]{Remark}

\numberwithin{equation}{section}

\newcommand{\A}{\mathbb{A}}

\newcommand{\onto}{\xrightarrow[]{{}_{\!\!\textnormal{onto\,\,\,}\!\!}}}
\newcommand{\into}{\xrightarrow[]{{}_{\!\!\textnormal{into\,\,\,}\!\!}}}

\def\XXint#1#2#3{{\setbox0=\hbox{$#1{#2#3}{\int}$}\vcenter{\hbox{$#2#3$}}\kern-.5\wd0}}

\def\XXiint#1#2#3{{\setbox0=\hbox{$#1{#2#3}{\iint}$}\vcenter{\hbox{$#2#3$}}\kern-.5\wd0}}

\begin{document}

\title[An Essay on the Interpolation Theorem of J\'{o}zef Marcinkiewicz\;\;\;]{An Essay on the Interpolation Theorem of \;\\J\'{o}zef Marcinkiewicz - Polish Patriot}

\author[Iwaniec]{Tadeusz Iwaniec}
\address{Department of Mathematics, Syracuse University, Syracuse,
NY 13244, USA and Department of Mathematics and Statistics,
University of Helsinki, Finland}
\email{tiwaniec@syr.edu}

\thanks{The author was supported by the NSF grant DMS-0800416 and the Academy of Finland project 1128331.}

%    General info

\date{\today}


\maketitle
\begin{center}
\textit{In memory of Polish mathematicians\\murdered by  the Soviets and the Nazis\\}\end{center}

\vskip2cm

\textit{\textbf{Prologue}}

\begin{quote}

"You will have to name it,"  Pierre said to his young wife, in the same tone as if it were a question of choosing a name for little Ir\`{e}ne.\\
The one-time Mlle Sklodovska reflected in silence for a moment. Then, her heart turning toward her own country which had been erased from the map of the world, she wondered vaguely if the scientific event would be published in Russia, Germany and Austria- the oppressor countries-and answered timidly:\\
"Could we call it 'polonium' ? "\\
In the Proceedings of the Academy of Science for July 1898 we read:
" If the existence of this new metal is confirmed we propose to call it \textit{polonium}, from the name of the original country of one of us."
\end{quote}

\begin{flushright} -from the book MADAME CURIE \\A Biography by \`{E}ve Curie. \\The Literary Guild of America,\\ INC. New York  1937  (page 161)
\end{flushright}
\vskip0.2cm
In May, 1921, President Harding presented Maria Sk{\l}odowska with one gram of radium. Later, in 1929, she donated a second gram of radium to help build The Radium Institute in Warsaw (1932).\\

Maria Sk{\l}odowska's life, regardless of her fame,  has immensely inspired fellow generations of scientists in Poland and abroad. J\'{o}zef Marcinkiewicz was about two years old when she was awarded her second Nobel Prize (in chemistry, December 10, 1911).  This year we will celebrate  the one hundredth anniversary of this event.\\

In 1939, following their secret protocol (signed by Molotov and Ribbentrop in August 23, 1939), Adolf Hitler (on September 1) and Joseph Stalin (on September 17) attacked Poland hoping to erase it from the map of the world again.
J\'{o}zef Marcinkiewicz, like Maria Sk{\l}odowska forty years before, stood up for his beloved country. In August 1939, when the Second World War was imminent, he came back from London to Wilno (Vilnius). He put on  Polish military uniform to say no to the Nazis and the Bolsheviks.   \\
 \begin{quote}"As a patriot and son of my homeland I would never attempt to refuse the service to the country in such difficult time as war".
\end{quote}
 \vskip0.2cm
 \begin{flushright}  - fragment of a letter of J\'{o}zef Marcinkiewicz\\to his adviser Antoni Zygmund, see \cite{Mal}
 \end{flushright}

Marcinkiewicz, along with  22 thousand Polish army officers, police members, land owners -great patriots who dared to exhibit a love and pride of independent Poland, were executed by NKVD murderers. By the order of J. Stalin, they were shot in the back of the head and buried secretly in mass graves of gloomy forested sites near Starobielsk, Ostashkovo and the most documented Katy\'{n}.

\begin{quote}
 \textit{KATY\'{N}  CAROL}

\textit{Someday maybe a great musician will rise up,\\
will transform speechless rows of gravestones into a keyboard,\\
a great Polish song writer will compose a frightening ballad with blood and tears.\\
$[...]$\\
And there will emerge untold stories,\\
strange hearts, bodies bathed in light...\\
And the Truth again will embody\\
The Spirit\\
with living words-of the sand of Katy\'{n}} \\

\end{quote}
\begin{flushright} - \textit{Kazimiera I{\l}{\l}akowicz\'{o}wna} \\
(translated by the author of this article)
\end{flushright}

\vskip0.5cm

\begin{quote}
"...his early death may be seen as a great blow to Polish Mathematics, and probably its heaviest individual loss during the Second World War".\\
\end{quote}
\begin{flushright} - remark by Professor Antoni Zygmund \\ about his pupil J\'{o}zef Marcinkiewicz \cite{DH, M2, Z2}
 \end{flushright}
\vskip0.2cm
Good scientific communities, like families, honor the memory of their eminent members, great patriots and martyrs.
I devote my essay to the memory of J\'{o}zef Marcinkiewicz  and all Polish mathematicians whose glorious scientific careers had come to a cruel end during Nazi-Soviet occupation.  J\'{o}zef Marcinkiewicz, Stanis{\l}aw Saks and  Juliusz Pawe{\l} Schauder were inspiration to me. I am mindful of them not only as mathematicians.\\

\textbf{\textit{Stanis{\l}aw Saks}} (1897-1942) was born to a Polish-Jewish family. He joined the Polish underground, was arrested and executed in November 23, 1942, by the German Gestapo in Warsaw.\\

Words written on the wall of a cell by an anonymous prisoner of the Gestapo in Aleja Szucha in Warsaw can be translated as saying:

\begin{flushright}\textit{Speaking of Poland is easy}\\
\textit{working for it is harder}\\
\textit{dying is harder still}\\
\textit{but suffering is the hardest.}
\end{flushright}
\vskip0.2cm
\textbf{\textit{Juliusz Pawe{\l} Schauder }}(1899 -1943), a Polish mathematician of Jewish origin who was shot in Lw\'{o}w by the Gestapo in September 1943. Immediately after that,  his wife Emilia and her daughter Ewa were hiding in the sewers. They eventually surrendered to the Gestapo. Transported to the concentration camp in Lublin, Emilia died, her daughter Ewa survived the camp.  \\

In August 1944 the staff of The Radium Institute in Warsaw  suffered the same fate as the victims of the Katy\'{n} massacre, they were executed by a shot in the back of the head. After the Second World War The Institute was  named  "Maria Sk{\l}odowska-Curie Institute of Oncology".\\

\textit{Acknowledgement}. I feel highly honored by the invitation to speak at the \textbf{\textit{J\'{o}zef Marcinkiewicz Centenary Conference}} in Pozna\'{n}, the city of the poet Kazimiera I{\l}{\l}akowicz\'{o}wna, and the city where Marcinkiewicz got his last, unrealized, offer to work (August 1939).\\
 \begin{flushright}\textbf{\textit{Heart-felt thanks to the organizers.}}
\end{flushright}

\vskip1cm

 \tableofcontents

\section{Introduction}

The total record of accomplishments of Marcinkiewicz in his short life, his talent, perceptions rich in concepts, and technical novelties, go far beyond my ability to give full play within the confines of one article.  The importance of Marcinkiewicz's short paper \cite{M1}, see also \cite {Z1}, is reflected in the myriad applications and generalizations \cite{BS, BL, C1, C2, CCRSW, CN, KPS, St, Z1} ; which earns the right to be called
$$Marcinkiewicz \;\;Interpolation \;\;Theory$$
Marcinkiewicz interpolation theorem came after the celebrated  convexity theorem of M. Riesz \cite{R} and his student G.O. Thorin \cite{T} . These fundamental works by M. Riesz, G.O. Thorin and J. Marcinkiewicz  deal with  estimates of  the $\,\mathscr L^p$-norms of an operator knowing its behavior at the end-points of the interval of the exponents $\,p\,$, where the operator is still defined.   There are, however, some subtle differences between the Riesz-Thorin and the Marcinkiewicz ideas.  Marcinkiewicz approach can be adapted to nonlinear operators, this is what we are going to demonstrate in the present paper. On the other hand, the very elegant idea of complex interpolation by G.O. Thorin \cite{T} has found numerous applications, especially when dealing with sharp inequalities for singular integrals.  The interested reader may wish to take a note of the interpolation lemmas in  \cite{AIPS}.\\
In the present paper I will try to elucidate some new advances of Marcinkiewicz  interpolation theorem which arise from a study of the nonlinear $\,p$-harmonic type PDEs,  \cite{DI1,DI2, DI3, GIS, I1, I2, I3}.  The principal result in this paper can be described as follows: \\

Let $\,(\mathbb X, \textnormal d x)\,$ be $\,\sigma$-finite measure space, $\,\mathscr L^2(\mathbb X,\,\mathbb V)\,$ the
space of square integrable functions valued in a finite dimensional inner product space $\,\mathbb V\,,$  and $\,\mathscr L_{_+}^2(\mathbb X,\,\mathbb V)\,$ a closed subspace of $\,\mathscr L^2(\mathbb X,\,\mathbb V)\,$. We study the orthogonal projection
$$
  \Pi\,;\; \mathscr L^2(\mathbb X,\,\mathbb V)\,\onto \,\mathscr L_{_+}^2(\mathbb X,\,\mathbb V)\,
$$
The standing assumption about the subspace $\,\mathscr L_{_+}^2(\mathbb X,\, \mathbb V)\,$ is that $\,\Pi\,$ extends as a continuous linear operator,

$$
  \Pi\,;\; \mathscr L^s(\mathbb X,\,\mathbb V)\,\onto \,\mathscr L_{_+}^s(\mathbb X,\,\mathbb V)\,
$$
 to some range of exponents $\, s \in (s_1 \,,\, s_\infty\,)\,$, where $\,1 \leqslant s_1 < 2 < s_\infty \leqslant \infty\,$. For example, the $\,\mathscr L^2$ -projection of vector fields in $\,\mathbb R^n\,$ onto gradient fields is represented by the Riesz transforms, so $\,s_1 = 1\,$ and $\,s_\infty = \infty$.\\
 Now, choose and fix $\, p \in (\,s_1, s_\infty)\,$. The $\,\mathscr L^p$ -projection of $\,\mathfrak f \in \mathscr L^p(\mathbb X, \mathbb V)\,$ onto the space $\, \mathscr L_{_+}^p(\mathbb X, \mathbb V)\,$ is defined to be a function   $\, \alpha \in \mathscr L^p(\mathbb X, \mathbb V)\,$ that is the closest, in $\,\mathscr L^p\,$ -distance, to $\,\mathfrak f\,$. This gives rise to a nonlinear continuous operator
 $$
  \mathfrak R_p \,;\; \mathscr L^p(\mathbb X,\,\mathbb V)\,\onto \,\mathscr L_{_+}^p(\mathbb X,\,\mathbb V)\,
$$

This operator can be viewed as $\,p$-harmonic variant of the Riesz transforms. The space $\,\mathscr L^p(\mathbb X,\,\mathbb V)\,$ is the natural domain of definition of $\, \mathfrak R_p\,$. However, we are interested in the action of $\,\mathfrak R_p\,$ on spaces different from $\,\mathscr L^p(\mathbb X,\,\,\mathbb V)\,$; namely,
$$
  \mathfrak R_p \,;\; \mathscr L^s(\mathbb X,\,\mathbb V)\,\onto \,\mathscr L_{_+}^s(\mathbb X,\,\mathbb V)\,, \;\;\;\textnormal {with some exponents }\;\;\; s \neq p
$$
 \begin{theorem}[Interpolation of the $\mathscr L^p$-projections] \label{0}
Suppose
 $$
  \mathfrak R_p \,;\; \mathscr L^{\,r_i} (\mathbb X,\,\mathbb V)\,\onto \,\mathscr L_{\textnormal{weak}}^{\,r_i}(\mathbb X,\,\mathbb V)\,,\; i\in\{1, 2\}\, , \;\;\textnormal {where}\;\;s_1 \leqslant r_1 < r_2 \leqslant s_\infty
$$
This means that for each $\,i\in\{1, 2\}\,$ there is a constant $\, C_i\,$ such that
$$
  \textnormal{meas} \,\{\,x \in \mathbb X\,;\; |\, \mathfrak R_p\mathfrak f (x)\,| \;>\, t \,\} \;\leqslant C_i\, t^{- r_i} \,\int_\mathbb X |\,\mathfrak f(x) \,| ^{r_i} \,\textnormal d x\,,\quad \textnormal{for all}\;\; t > 0
$$
whenever $\,\mathfrak f \in  \mathscr L^{\,p} (\mathbb X,\,\mathbb V) \cap \mathscr L^{\,r_i} (\mathbb X,\,\mathbb V)\,$.\\
Then for every $\,r \in (\,r_1,\, r_2\,)\,$ there exists a constant  $\,C_r\,$ such that
$$
 \int_\mathbb X |\, \mathfrak R_p\mathfrak f(x)\,|^r \, \textnormal d x \; \leqslant  \; C_r\,\int_\mathbb X |\,\mathfrak f(x)\,|^r\,\textnormal d x
$$
whenever $\,\mathfrak f \in  \mathscr L^{\,p} (\mathbb X,\,\mathbb V) \cap \mathscr L^{\,r} (\mathbb X,\,\mathbb V)\,$.
\end{theorem}

The proof is immediate from a more general result that we included in Theorem \ref{Marcinkiewicz}, see Section 5 and Section 6.
 \vskip1cm

\section{A Motivation from Hodge Theory}

The Hodge decomposition of differential forms  provokes a nonlinear setting of Marcinkiewicz Interpolation.
\subsection{The linear Hodge theory}

Let $\mathbb X$ be an oriented  Riemannian $n$-dimensional smooth manifold (with or without boundary)  \cite{Du, I5, ISS, Mor, Sc}. To every point $\,x \in \mathbb X $ and  $ 0 \leqslant \ell \leqslant n $ there corresponds a linear space $\,\bigwedge^\ell_x\,$ of $\ell$-covectors.  This is an $\, n \choose  \ell $ - dimensional inner product space:

$$
 \langle\,\alpha \,|\, \beta\,\rangle \,\textnormal d x= \langle\,\alpha \,|\, \beta\,\rangle _{_x} \,\textnormal d x =  \alpha \wedge \ast \beta
$$
where  $\textnormal d x$ stands for the unit $n$-covector induced by the metric and orientation of the manifold, and $\,\ast\, ; \;\bigwedge^\ell_x\, \rightarrow \bigwedge^{n-\ell} _x\,$ is the Hodge-star duality operator. The volume element $\textnormal d x$ gives rise to a measure on $\,\mathbb X\,$.

Let $\,\bigwedge^\ell\, = \,\bigwedge^\ell(\mathbb X) = \bigcup _{x \in \mathbb X} \bigwedge^\ell_x\,$ denote the boundle of $\ell$- covectors over $\mathbb X$. One might consider various classes of sections of this bundle; that is, differential forms. The class of smooth $\ell$-forms will be denoted by $\mathscr C^\infty (\mathbb X , \wedge^\ell) $. There are two underlying differential operators acting on forms: the exterior differential   and its formal adjoint, called Hodge codifferential
$$ \textnormal d \,; \mathscr C^\infty (\mathbb X , \wedge^{\ell -1}) \rightarrow \mathscr C^\infty (\mathbb X , \wedge^\ell)\;, \quad \quad  \textnormal d^ \ast ; \mathscr C^\infty (\mathbb X , \wedge^{\ell +1}) \rightarrow \mathscr C^\infty (\mathbb X , \wedge^\ell)$$
The Hodge theory asserts that every differential $\,\ell$-form $\,\omega \in \mathscr C^\infty (\mathbb X , \wedge^\ell)\, $ can be written as
$$
  \omega  = \textnormal d u \, + \, \textnormal d^*  v  +  \, h \,, \quad   u \in \mathscr C^\infty (\mathbb X , \wedge^{\ell -1})\,, \quad v \in \mathscr C^\infty (\mathbb X , \wedge^{\ell +1})\,, \quad    h \in \mathscr C^\infty (\mathbb X , \wedge^{\ell})
$$
The exact component $\, \textnormal d u \,$, the coexact component $\, \textnormal d^* v  \,$, and the harmonic field $\,h\,,\; \textnormal d h = \textnormal d^* h = 0 \,$, are determined uniquely once we impose suitable boundary conditions on  $\,u\,, v\,, h\,$ (no conditions are required for compact manifolds without boundary), see \cite{Du, I5, ISS, Mor}. Under such boundary conditions these components are mutually orthogonal in the space
$$
 \mathscr L^2 ( \mathbb X, \wedge ^\ell) = \{ \,\,\omega \,; \;\|\omega \|_{_2}^2 = \int_\mathbb X \omega \wedge \ast \omega  = \int _\mathbb X \langle \,\omega(x)\,|\,\omega(x) \,\rangle\,\textnormal d x \, < \infty\,\}
$$
In fact this is the space where the Hodge decomposition theory is born. Let us look briefly at the following two closed subspaces of  $\,\mathscr L^2 ( \mathbb X, \wedge ^\ell)\,$; those that consist of the exact and coclosed forms, respectively:
$$
 \mathscr L^2_{_+} (\mathbb X\,,\wedge^\ell) = \mathscr L^2-\textnormal{closure of the forms }\;\textnormal d u\,, \;\textnormal{\;with }\;  u \in \mathscr C^\infty _\circ(\mathbb X\,,\wedge ^{\ell-1}\,)
$$
$$
 \mathscr L^2_{_-} (\mathbb X\,,\wedge ^\ell) = \mathscr L^2-\textnormal{closure of the forms }\; \beta \in \mathscr C^\infty (\mathbb X\,, \wedge ^\ell)\,\, \textnormal{such that }\;\textnormal d^* \beta = 0
$$
Thus we have an orthogonal decomposition
$$
\mathscr L^2 ( \mathbb X, \wedge ^\ell)\, =  \mathscr L^2_{_+} (\mathbb X\,,\wedge^\ell)\; \oplus\;  \mathscr L^2_{_-} (\mathbb X\,,\wedge ^\ell)
$$
Let
$$
\mathbf E \;;\;  \mathscr L^2 ( \mathbb X, \wedge ^\ell)\, \rightarrow \mathscr L^2_{_+} (\mathbb X\,,\wedge^\ell)\;\quad\textnormal{and}\quad
\mathbf C \;;\;  \mathscr L^2 ( \mathbb X, \wedge ^\ell)\, \rightarrow  \mathscr L^2_{_-} (\mathbb X\,,\wedge^\ell)
$$
denote the orthogonal projections. These operators are locally represented by singular integrals (Riesz Transforms) and as such keep acting   as  continuous operators on every $\,\mathscr L ^s$ -space, with $ \,1 < s < \infty $, \cite{I5, ISS, Mor, Sc}
$$
\mathbf E \;;\;  \mathscr L^s ( \mathbb X, \wedge ^\ell)\, \rightarrow \;\mathscr L^s_{_+} (\mathbb X\,,\wedge^\ell)\quad\textnormal{and}\quad
\mathbf C \;;\;  \mathscr L^s ( \mathbb X, \wedge ^\ell)\, \rightarrow  \mathscr L^s_{_-} (\mathbb X\,,\wedge^\ell)
$$
Let us record the following generalization of the orthogonality  of  exact and coclosed forms
$$
 \int _\mathbb X \langle\, \alpha(x)\,|\,\beta(x)\,\rangle\,\textnormal d x \;=\;0\;,  \quad\quad\textnormal{whenever}\;\; \alpha \in \mathscr L^p_{_+} ( \mathbb X, \wedge ^\ell)\,\quad\textnormal{and}\;\;\beta \in \mathscr L^q_{_-} ( \mathbb X, \wedge ^\ell)\,
$$
Hereafter $\,p\,$ and $\,q\,$ are H\"{o}lder conjugate exponents; that is, real numbers in the interval $ (1\,, \, \infty \,) $ that satisfy  the H\"{o}lder relation $\,p + q = p\cdot q \,$.

These projection operators also act on $ \, \mathscr L^1 ( \mathbb X, \wedge ^\ell)\,$ with values, respectively,  in exact and coclosed forms of the Marcinkiewicz class  $\,\mathscr L^1_{\textnormal{weak}} ( \mathbb X, \wedge ^\ell)\,$.

\subsection{The $\mathscr L^p$- projection}
 The quintessential problem  is to find an exact  differential form $\, \alpha \in \mathscr L^p_{_+} ( \mathbb X, \wedge ^\ell)\, , \,p \neq \,2\,$,  which is the nearest possible in the $\mathscr L^p$-distance to a given form $\,\omega \in \mathscr L^p ( \mathbb X, \wedge ^\ell)\,\,$. In the pursuit of the solution we minimize the $\,p$-harmonic energy integral
 $$
     \min_{\gamma \in \mathscr L^p_{_+} ( \mathbb X,\, \wedge ^\ell)}\int_\mathbb X \,| \omega(x) - \gamma (x) \,|^p \;\textnormal d x =
     \int_\mathbb X \,|\, \omega(x) - \alpha (x) \,|^p \;\textnormal d x
 $$
Such form $\alpha \in \mathscr L^p_{_+} (\mathbb X, \wedge ^\ell)$ solves  the Euler-Lagrange equation
\begin{equation}
 \textnormal d^* |\omega - \alpha | ^{p-2} (\omega - \alpha )   =  0 \,,
\end{equation}
Equivalently,
\begin{equation}
 |\omega - \alpha | ^{p-2} (\omega - \alpha ) {\overset{{}_{\textnormal{\tiny{def}}}}{\;=\!\!=\;}} \beta  \in \mathscr L^q_{_-} (\mathbb X, \wedge ^\ell)
\end{equation}
Here, we are dealing with H\"{o}lder conjugate exponents,  $\, p + q = p\,\cdot\, q \,$,  and the Hodge codifferential $\,\textnormal d^* \,; \mathscr L^q (\mathbb X\,, \wedge ^\ell) \,\rightarrow  \mathscr D\,' (\mathbb X\,, \wedge ^{\ell - 1 }) \;$ acting in the sense of distributions. Now, by the very definition, the $\,\mathscr L^p$-projection onto exact forms is a nonlinear operator which takes $\,\omega\,$ into $\,\alpha\,$,
\begin{equation}
 \mathbf E^p \;;\;\mathscr L^p (\mathbb X\,, \wedge ^\ell)\; \onto  \;  \mathscr L^p_{_+} (\mathbb X\,, \wedge ^\ell)\;,  \quad \mathbf E^p \omega = \alpha
\end{equation}
It is plain that the operator $\,\mathbf E^p \,$ is bounded,
\begin{equation}
 \|\,\mathbf E^p\omega\|_p \; \leqslant \| \omega \|_p
\end{equation}
But it is less obvious, because of nonlinearity, that $\,\mathbf E^p\,$ is also continuous; precisely, we have:
\begin{equation}
\int_\mathbb X |\, \mathbf E^p \omega_1\;-\; \mathbf E^p\omega_2\,|^p \;\ll \Big( \int_\mathbb X |\omega _1\,-\, \omega_2 \,|^p \Big )^{\theta}\,\cdot\, \Big( \int_\mathbb X |\omega _1\,|^p\; +\, |\,\omega_2 \,|^p \Big )^{1 -\theta}
\end{equation}
for some exponent $\,0 < \theta = \theta(p) \leqslant 1\,$.
Throughout this text we shall make use of the symbol $ \,\ll\,$ to indicate that the inequality holds with some positive constant, so-called \textit{implied constant},  in front of the expression after this symbol. The implied constants will vary from line to line; their detailed dependence on the exponents and other parameters can easily be perceived from the computation. The implied constants will never depend on the functions of concern.\\
While the space $\,\mathscr L^p (\mathbb X\,, \wedge ^\ell)\,$ is considered the natural domain of definition of the operator $\,\mathbf E^p\,$ we shall depart from this space and move into the realm of exponents different from $\,p$. A wider and unifying framework will be set up for such operators to capture the essence of Marcinkiewicz  interpolation.

\section{Setting the Stage}
From now on $\,(\mathbb X\,, \textnormal d x\,)\,$ will be an arbitrary $\,\sigma$-finite measure space and  $\,\mathbb V\,$ a finite dimensional vector space equipped with an inner product $\, \langle\;\cdot|\:\cdot \rangle =  \langle\;\cdot|\: \cdot\rangle_{_\mathbb V}\;$ and the induced norm $\,|\cdot | = |\cdot |_{_\mathbb V} = \langle\;\cdot|\: \cdot\rangle^{1/2}$. We shall discuss $\,\mathscr L^s$-spaces, $\,1\leqslant s < \infty $,  of $\,\textnormal d x$-measurable functions on $\mathbb X\,$ valued in $\mathbb V\,$,

$$
\mathscr L^s(\mathbb X\,, \mathbb V) \;=\; \Big \{\,\omega \;;\; \|\,\omega\,\|_{_s} \;=\; \Big (\int_\mathbb X |\,\omega(x)\,|_{_\mathbb V} ^s \;\textnormal d x \;\Big )^{\frac{1}{s}}  <\infty \;\Big \}
$$
\begin{remark}
The observant reader may wish to note that our subsequent considerations remain valid for functions in $\,\mathbb X\,$ whose values $\,\omega(x)\,$  lie in a given inner product space $\,\mathbb V_x\,$  assigned to each point $\, x \in \mathbb X\,$, like differential forms on Riemannian manifolds. Rigorous treatment of this setting, however, would require introduction of relevant vector bundles over $\,\mathbb X\,$, so that $\,\omega \,$ would become a measurable section. For the sake of readability, we do not enter this territory; instead, we  confine ourselves to trivial bundle $\,\mathbb X \times \mathbb V\,$.
\end{remark}
\subsection{The $\mathscr L^2$-projections}
Let us choose and fix a closed subspace  $ \mathscr L^2_{_+}(\mathbb X) = \mathscr L^2_{_+}(\mathbb X \,, \mathbb V)\,\subset  \mathscr L^2(\mathbb X \,, \mathbb V)\,$ and its orthogonal complement  $\,\mathscr L^2_{_-}(\mathbb X) = \mathscr L^2_{_-}(\mathbb X\,, \mathbb V) \subset  \mathscr L^2(\mathbb X \,, \mathbb V)$. Thus, we have an orthogonal decomposition
\begin{equation} \label{decomposition} \mathscr L^2(\mathbb X) \, =  \,\mathscr L^2_{_+}(\mathbb X) \,\oplus \,\mathscr L^2_{_-}(\mathbb X)
\end{equation}
and the induced $\,\mathscr L^2$-projections
$$
 \Pi_{_+} \,; \mathscr L^2(\mathbb X) \, \onto\;   \,\mathscr L^2_{_+}(\mathbb X)\;,\quad \; \Pi_{_-} \,; \mathscr L^2(\mathbb X) \, \onto\;   \,\mathscr L^2_{_-}(\mathbb X)\;,\quad \;\mathbf I = \Pi_{_+}\,+\, \Pi_{_-}
$$
In general these linear operators may, or may not, extend to any of  $\,\mathscr L^s$-spaces with $\,s \neq 2\,$. Moreover, if they do extend, the range of such exponents depends on the decomposition (\ref{decomposition}). Let us take for granted an assumption that
\begin{equation}
\Pi_{_+} \,; \mathscr L^\mathfrak q (\mathbb X) \, \into\;   \,\mathscr L^\mathfrak q _\textnormal{weak}(\mathbb X)\,, \quad \textnormal{for some}\;\; 1\leqslant \mathfrak q < 2
\end{equation}
 meaning that for every $\,t > 0\,$ it holds
 \begin{equation}
 \textnormal{meas}\{ \,x\,; \;|\Pi_{_+}\omega (x)| > t\;\} \;\ll \frac{1}{t^\mathfrak q} \int_\mathbb X |\,\omega\,|^\mathfrak q\,,\quad \textnormal{whenever}\;\;\;\omega \in \mathscr L^2(\mathbb X)\cap \mathscr L^\mathfrak q(\mathbb X)
 \end{equation}
where the implied constant depends neither on $\,t\,$ nor on the function $\,\omega\,$. Of course the same is true for the operator $\,\Pi_{_-}\,$. By virtue of Marcinkiewicz Interpolation Theorem these two projections  extend as continuous operators:
\begin{equation}
\Pi_{_+} \,; \mathscr L^s(\mathbb X) \, \onto\;   \,\mathscr L^s_{_+}(\mathbb X)\,, \quad \textnormal{for all}\;\; \mathfrak q <s \leqslant 2
\end{equation}
and
\begin{equation}
\Pi_{_-} \,; \mathscr L^s(\mathbb X) \, \onto\;   \,\mathscr L^s_{_-}(\mathbb X)\,, \quad \textnormal{for all}\;\; \mathfrak q <s \leqslant 2
\end{equation}
where the spaces $\,\mathscr L^s_{_+}(\mathbb X)\,$ and  $\,\mathscr L^s_{_-}(\mathbb X)\,$  are the closures of  $\,\mathscr L^2_{_+}(\mathbb X) \cap \,\mathscr L^s(\mathbb X)\,$\,and \,$\,\mathscr L^2_{_-}(\mathbb X) \cap \,\mathscr L^s(\mathbb X)\,$\, in the norm of $\,\mathscr L^s(\mathbb X)\,$, respectively.
It is rather intriguing that a direct application of Marcinkiewicz interpolation does not guarantee a uniform bound of the norms of  $\,\Pi_{_\pm}\,$ as $\, s \,$ approaches 2, though such uniform bounds are evident from the Riesz-Thorin convexity theorem.  Nevertheless, since the  operators $\,\Pi_{_\pm}\,$ are selfadjoint in $\,\mathscr L^2(\mathbb X\,,\mathbb V)\,$, we may appeal to H\"{o}lder's duality to infer that
\begin{equation}
\Pi_{_+} \,; \mathscr L^s(\mathbb X) \, \onto\;   \,\mathscr L^s_{_+}(\mathbb X)\,, \quad \textnormal{for all}\;\; 2\leqslant s < \mathfrak p
\end{equation}
and
\begin{equation}
\Pi_{_-} \,; \mathscr L^s(\mathbb X) \, \onto\;   \,\mathscr L^s_{_-}(\mathbb X)\,, \quad \textnormal{for all}\;\; 2\leqslant s < \mathfrak p
\end{equation}
where $\,1\leqslant \mathfrak q < 2 < \mathfrak p \leqslant \infty\,$ is a H\"{o}lder conjugate pair; that is,  $ \mathfrak q \,+\, \mathfrak p\, =\, \mathfrak q \,\cdot\,\mathfrak p \,$.

Now, a uniform bound of the $\,s$-norms with  $\,s\approx 2\,$ follows from the Marcinkiewicz interpolation theorem as well. Let us summarize the above findings as
\begin{equation}
\Pi_{_\pm} \,; \mathscr L^s(\mathbb X) \, \onto\;   \,\mathscr L^s_{_\pm}(\mathbb X)\,, \quad \textnormal{for all}\;\; \mathfrak q \,<\, s \, <\,\mathfrak p
\end{equation}
The spaces  $\,\mathscr L^s_{_+}(\mathbb X)\,$ and  $\,\mathscr L^s_{_-}(\mathbb X)\,$ can also be characterized through the following properties
$$
\mathscr L^s_{_+}(\mathbb X) = \{\,\alpha \in \mathscr L^s(\mathbb X) \,; \;\Pi_{_+} \alpha = \,\alpha\,\}\,,\quad\; \mathscr L^s_{_-}(\mathbb X) = \{\,\beta \in \mathscr L^s(\mathbb X)\,; \;\Pi_{_-} \beta = \,\beta\,\}
$$
From now on we reserve the letters $\,p\,,\,q\,$ for a pair of H\"{o}lder conjugate exponents in the range
$$
 \mathfrak q <  q \leqslant p < \mathfrak p \,\,,\quad
 \textnormal{where}\;\;\;\;\; \frac{1}{p} \,+\, \frac{1}{q}\,= 1\,
$$
Let us conclude this subsection by recording  the  \textit{$(p,q)$- orthogonality } equation, reminiscent of the familiar div-curl product equation \cite{CLMS, I4, I5, IS2, Mu}
\begin{equation}\label{orthogonality}
 \int_\mathbb X \langle\, \alpha (x)\;|\; \beta(x)\;\rangle\;\textnormal d x  \;=\; 0\,, \quad\;\textnormal{for}\;\;\alpha \in \mathscr L^p_{_+}(\mathbb X) \;\;\;\textnormal{and}\; \;\; \beta \in \mathscr L^q_{_-} (\mathbb X)
\end{equation}
This shows that the dual space to $\,L^p_{_+}(\mathbb X)\,$ is $\,L^q_{_+}(\mathbb X)\,$, and  the dual space to  $\,L^p_{_-}(\mathbb X)\,$ is $\,L^q_{_-}(\mathbb X)\,$. In symbols, we have
$$
\,L^p_{_+}(\mathbb X)^{^\bigstar} \,=\, \,L^q_{_+}(\mathbb X)\,,\quad \,L^q_{_+}(\mathbb X)^{^\bigstar} \,=\, \,L^p_{_+}(\mathbb X)\,$$
$$L^p_{_-}(\mathbb X)^{^\bigstar} \,=\, \,L^q_{_-}(\mathbb X)\,,\quad \,L^q_{_-}(\mathbb X)^{^\bigstar} \,=\, \,L^p_{_-}(\mathbb X)
$$
Thus all the above are reflexive Banach spaces.
\subsection{The $\Pi_{_+}^p$ and $\Pi_{_-}^q$ projections} To every $\mathfrak a \in \mathscr L^p(\mathbb X\,,\mathbb V)\,$ there corresponds exactly one $\,\alpha \in \mathscr L^p _{_+}(\mathbb X\,,\mathbb V)\,$ which is the nearest to $\,\mathfrak a\,$ in the sense of $\,\mathscr L^p$-distance.
$$
\int_\mathbb X |\,\mathfrak a (x) \,-\, \alpha(x)\,|^p \,\textnormal d x  \;=\; \min_{\gamma \in \mathscr L^p _{_+}(\mathbb X\,,\mathbb V)} \;\int_\mathbb X |\,\mathfrak a (x) \,-\, \gamma(x)\,|^p \,\textnormal d x
$$
The variational equation for $\,\alpha\,$ takes the form
\begin{equation}
 |\,\mathfrak a \, - \,\alpha | ^{p-2} (\,\mathfrak a\, - \,\alpha )\;  \in\; \mathscr L^q_{_-} (\mathbb X, \mathbb V)
\end{equation}
In this way we have defined a nonlinear operator
$$ \Pi_{_+}^p \;;\; \mathscr L^p(\mathbb X\,,\mathbb V)\, \onto\; \mathscr L^p_{_+} (\mathbb X\,,\mathbb V)\,,\quad\,\textnormal{ by the rule}\,:\;\;\Pi_{_+}^p\mathfrak a = \alpha.
$$
Similarly, to every $\mathfrak b \in \mathscr L^q(\mathbb X\,,\mathbb V)\,$ there corresponds exactly one $\,\beta \in \mathscr L^q _{_-}(\mathbb X\,,\mathbb V)\,$ which is the $\mathscr L^q$-nearest to $\,\mathfrak b\,$.
The analogous variational equation for $\,\beta\,$ takes the form
\begin{equation}
 |\,\mathfrak b \, - \,\beta | ^{q-2} (\,\mathfrak b\, - \,\beta  ) \;  \in \;\mathscr L^p_{_+} (\mathbb X, \mathbb V)
\end{equation}
and defines a nonlinear operator
$$ \Pi_{_-}^q \;;\; \mathscr L^q(\mathbb X\,,\mathbb V)\, \onto\; \mathscr L^q_{_-} (\mathbb X\,,\mathbb V)\,,\quad\,\textnormal{ by the rule}\,:\;\;\Pi_{_-}^q\mathfrak b = \beta.
$$
Let us weld together the above variational equations by introducing so-called $\,(p\,,q)$-system
\begin{lemma} Given a pair $\,\mathfrak f = (\mathfrak a\,,\,\mathfrak b\,) \in \mathscr L^p(\mathbb X\,,\mathbb V)\times \mathscr L^q(\mathbb X\,,\mathbb V)\,$ there exists exactly one pair
$\,\phi = (\alpha\,,\,\beta\,) \in \mathscr L^p_{_+}(\mathbb X\,,\mathbb V)\times \mathscr L^q_{-}(\mathbb X\,,\mathbb V)\,$  such that
\begin{equation}\label{pq-equation}
 |\,\mathfrak a \, - \,\alpha | ^{p-2} (\,\mathfrak a\, - \,\alpha )\;=\; \mathfrak b - \beta   \in\; \mathscr L^q_{_-} (\mathbb X, \mathbb V)
\end{equation}
 or, equivalently
 \begin{equation}\label{qp-equation}
 |\,\mathfrak b \, - \,\beta | ^{q-2} (\,\mathfrak b\, - \,\beta\,)\;=\; \mathfrak a - \alpha   \in\; \mathscr L^p_{_+} (\mathbb X, \mathbb V)
\end{equation}
\end{lemma}
\begin{proof}
One finds, uniquely, $ \,\alpha \in  \mathscr L^p_{_+} (\mathbb X\,,\mathbb V)\,$ and $\, \beta \in  \mathscr L^q_{_-} (\mathbb X\,,\mathbb V)\,$ by solving   the following strictly convex variational problem: given a pair $\, \mathfrak f = (\mathfrak a\,,\,\mathfrak b ) \in   \mathscr L^p (\mathbb X\,,\mathbb V)\,\times\,   \mathscr L^q (\mathbb X\,,\mathbb V)\,$ find $ \,\phi =(\alpha\,,\, \beta )  \in  \mathscr L^p_{_+} (\mathbb X\,,\mathbb V)\,\times \, \mathscr L^q_{_-} (\mathbb X\,,\mathbb V)\,$ such that
$$
 \min_{\alpha'\in \mathscr L^p _{_+}(\mathbb X\,,\mathbb V)} \;\int_\mathbb X |\,\mathfrak a \,-\, \alpha'\,|^p \; -\; p \,\langle\,\mathfrak b\,|\, \alpha' \rangle\;=\; \int_\mathbb X |\,\mathfrak a \,-\, \alpha\,|^p \; -\; p \,\langle\,\mathfrak b\,|\, \alpha \rangle
$$
or, equivalently,
$$
 \min_{\beta' \in \mathscr L^q _{_-}(\mathbb X\,,\mathbb V)} \;\int_\mathbb X |\,\mathfrak b \,-\, \beta'\,|^q \; -\; q \,\langle\,\mathfrak a\,|\, \beta' \rangle\;=\; \int_\mathbb X |\,\mathfrak b \,-\, \beta\,|^q \; -\; q \,\langle\,\mathfrak a\,|\, \beta\rangle $$
 \end{proof}
The reader may wish to keep an eye on a duality between these two  variational problems; precisely,  the solution of one of them yields the solution of the other via the equations
(\ref{pq-equation}) and (\ref{qp-equation}). In fact they can be given one aesthetically pleasing symmetric form,

\begin{equation}\label{power equation}
(\mathfrak a - \alpha)^ p \;=\; (\mathfrak b - \beta)^q  \;\in \mathscr L^1 (\mathbb X\,,\mathbb V) \,,
\end{equation}
Here we have introduced the notation $\,\mathfrak v ^s {\overset{{}_{\textnormal{\tiny{def}}}}{\;=\!\!=\;}} |\mathfrak v|^{s-1}\mathfrak v\,$ for the $\,s$-power of a vector $\,\mathfrak v\,$ in a normed space. \\
It is worthwhile to put the equation (\ref{power equation})  in even more general framework.
\subsection{The most general setting}
Let $\,\mathfrak A \,: \,\mathbb X \times \mathbb V \rightarrow \mathbb V \,$ be a given function satisfying the following requirements:
\begin{itemize}
\item  \textit{Carath\'{e}odory regularity}:\\
The function $\, x\rightarrow \mathfrak A(x\,, \mathfrak v )\,$ is measurable for every  $\mathfrak v \in \mathbb V $\\
The function $\,\mathfrak v \rightarrow \mathfrak A(x\,, \mathfrak v )\,$ is continuous for almost every  $\, x \in \mathbb X $
\item  \textit{Homogeneity}:
$$ \mathfrak A(x\,, \lambda \,\mathfrak v )  = \lambda^{p-1} \cdot \mathfrak A(x\,, \mathfrak v )\,,\quad\, \textnormal{for}\;\;\lambda \geqslant 0 $$
\item \textit{Lipschitz condition:}
$$
|\,\mathfrak A(x\,, \mathfrak v_1 ) \; -\; \mathfrak A(x\,, \mathfrak v_2 )\,| \;\ll\;  \big(\,|\mathfrak v_1\,|\,+\, |\mathfrak v_2\,|\,\big)^{p-2}|\,\mathfrak v_1 \,-\, \mathfrak v_2\,|
$$
\item \textit{Monotonicity condition}:
$$
\langle\,\mathfrak A(x\,, \mathfrak v_1 ) \; -\; \mathfrak A(x\,, \mathfrak v_2 )\;|\; \mathfrak v_1\,-\,\mathfrak v_2 \;\rangle \;\gg\;  \big(\,|\mathfrak v_1\,|\,+\, |\mathfrak v_2\,|\,\big)^{p-2}|\,\mathfrak v_1 \,-\, \mathfrak v_2\,|^2
$$
\end{itemize}
\begin{remark}
In the above inequalities the implied constants are independent of $\, x \in \mathbb X\, $ and $\,\mathfrak v_1,\,\mathfrak v_2\,\in \mathbb V\,$. The Carath\'{e}odory regularity makes certain that the function $\, x\rightarrow \mathfrak A(x\,, \mathfrak v(x) )\,$ is measurable whenever $\,\mathfrak v = \mathfrak v(x)\,$ is measurable, by Scorza-Dragoni Theorem.
\end{remark}
 It is clear that for $\,x\,$ fixed the mapping $\,\mathfrak v \rightarrow \mathfrak A(x\,, \mathfrak v )\,$ is invertible. Let its inverse be denoted by $
  \,\mathfrak v \rightarrow \mathfrak B(x\,, \mathfrak v )\,$; that is, for almost every $\,x\in\,\mathbb X\,$, we have
   $$\mathfrak B(x\,, \ast) \circ \mathfrak A(x\,, \ast ) \;=\; \mathfrak A(x\,, \ast) \circ \mathfrak B(x\,, \ast )\;=\;\mathbf I : \mathbb V \rightarrow \mathbb V
 $$
\subsection{The natural domain of definition} The problem of solving the most general equation in its natural domain of definition reads as follows. \\

\textbf{Problem 1}.
Given a pair $\, \mathfrak f = (\mathfrak a\,,\mathfrak b ) \in   \mathscr L^p (\mathbb X\,,\mathbb V)\,\times\,   \mathscr L^q (\mathbb X\,,\mathbb V)\,$ solve the following equation
\begin{equation}\label{general equation}
\mathfrak A(x, \mathfrak a +\alpha)  =  \mathfrak b + \beta \,, \quad\textnormal{equivalently} \quad \mathfrak B(x, \mathfrak b +\beta)  =  \mathfrak a + \alpha
\end{equation}
for  $ \,\phi =(\alpha\,,\, \beta )  \in  \mathscr L^p_{_+} (\mathbb X\,,\mathbb V)\,\times \, \mathscr L^q_{_-} (\mathbb X\,,\mathbb V)\,$.\\

A duality between the two equations at (\ref{general equation}) is emphasized by the fact that the inverse mapping  $\,\mathfrak v \rightarrow \mathfrak B(x\,, \mathfrak v )\,$ satisfies the same requirements as $\,\mathfrak v \rightarrow \mathfrak A(x\,, \mathfrak v )\,$, but  with the H\"{o}lder conjugate exponent $\,q\,$ in place of $\,p\,$. The key to the existence and uniqueness of the solutions lies in the $(p,q)$-orthogonality relation (\ref{orthogonality}) between the unknown functions $\,\alpha\,$ and $\,\beta\,$.
\begin{theorem}\label{existence and uniqueness}
The equation (\ref{general equation}) has unique solution. Moreover
\begin{equation}\label{basic estimate}
\int_{\mathbb X} |\,\alpha\,| ^p\,+\, |\,\beta\,|^q \;\ll\; \int_{\mathbb X} |\,\mathfrak a\,| ^p\,+\, |\,\mathfrak b\,|^q
\end{equation}
\end{theorem}

\begin{proof}
Since in general the equation (\ref{general equation}) is not arising from a minimization of a variational integral, the existence of the solutions cannot be established by a variational method. The Minty-Browder theory of monotone operators ~\cite{Br},~\cite {Mi} will come to the rescue.
Fix a function $\, \mathfrak a \in \,\mathscr L^p_{_+} (\mathbb X)\,$ and consider a map $\,\mathbf T\,$ from  the reflexive Banach space $\,\mathscr L^p_{_+} (\mathbb X )\,$  into its dual  $\,\mathscr L^q_{_+} (\mathbb X )\,$, defined by the rule
$$
 \mathbf T \alpha = \Pi^q_{_+} \mathfrak A (x, \mathfrak a + \alpha ) \,\in  \,\mathscr L^q_{_+} (\mathbb X)\,, \quad\,\textnormal{for}\;\;\;\;\alpha \in \mathscr L^p_{_+}(\mathbb X) .
$$
A routine application of the requirements for $\mathfrak A\,$ shows that this map is:
\begin{itemize}
\item \textit{continuous},
\item \textit{strictly monotone}; that is,
$$ \int_{\mathbb X} \langle \mathbf T \alpha_1 \, - \mathbf T \alpha_2 \,|\,\alpha_1 \,- \,\alpha_2 \;\rangle > 0 \;, \quad \textnormal{whenever}\;\;\;\alpha_1 \neq \alpha_2 \,\quad \textnormal{in }\;\;\;\;\; \mathscr L^p_{_+} (\mathbb X )$$
\item \textit{and coercive}; that is,
$$
 \lim_{\|\alpha\|_p \rightarrow \infty }\, \frac{\int_\mathbb X \langle\, \mathbf T \alpha  \,|\,\alpha  \,\rangle}{\|\,\alpha\,\|_p} \;=\;\infty
$$
\end{itemize}
\vskip0.3cm
The Minty-Browder theory asserts that $\,\mathbf T\,$ is bijective. Let $\,\mathfrak b \in \mathscr L^q (\mathbb X)\,$ be given, so we are given  an element $\, \Pi^q_{_+} \mathfrak b \in \mathscr L^q_{_+} (\mathbb X)\,$ . There is exactly one $\, \alpha \in \mathscr L^p_{_+} (\mathbb X)\,$ such that $\,\Pi^q_{_+} \mathfrak A (x, \mathfrak a + \alpha ) = \Pi^q_{_+} \mathfrak b\,$, meaning that $\, \mathfrak A (x, \mathfrak a + \alpha ) =  \mathfrak b\, + \beta\,$, where $\,\beta \in
\mathscr L^q_{_-} (\mathbb X)\,$.\\
A key step in obtaining estimate (\ref{basic estimate}) is the $\,(p,q)$-orthogonality of $\,\alpha \in \mathscr L^p_{_+}(\mathbb X)\,$ and $\,\beta \in \mathscr L^q_{_-}(\mathbb X) \,$. Using the properties imposed on $\,\mathfrak A\,$ and $\,\mathfrak B\,$ a computation runs as follows

 \[\begin{split} |\,\mathfrak a +\,\alpha\,|^p & \ll\;\langle\;\mathfrak A(x, \mathfrak a + \alpha)\,|\, \mathfrak a + \alpha \,\rangle \;=  \langle\;\mathfrak A(x, \mathfrak a + \alpha)\,|\, \mathfrak a\,\rangle  + \langle\, \mathfrak b + \beta)\,|\, \alpha \,\rangle \;\;\\& \ll
 |\,\mathfrak a + \alpha)\,|^{p-1} |\mathfrak a \,|  \;+  \langle\,\mathfrak a + \alpha\,|\, \mathfrak b \,\rangle  \,- \,\langle\, \mathfrak a\,| \mathfrak b \,\rangle\,  + \,\langle\, \alpha\,|\, \beta\,\rangle
\end{split}\]
 This yields $$ |\,\mathfrak a +\,\alpha\,|^p \ll |\,\mathfrak a\,|^p \,+\, |\,\mathfrak b\,|^q \,+  \,\langle\, \alpha\,|\, \beta\,\rangle\,$$
On the other hand, in view of (\ref{general equation}) it follows that
 $$|\,\mathfrak b \,+\,\beta \,|^q  \ll   |\,\mathfrak a +\,\alpha\,|^p  \ll\, |\,\mathfrak a\,|^p \,+\, |\,\mathfrak b\,|^q \,+  \,\langle\, \alpha\,|\, \beta\,\rangle $$
 Summing these two inequalities, we obtain
 \begin{equation}
 |\,\mathfrak a +\,\alpha\,|^p\;+\;|\,\mathfrak b \,+\,\beta \,|^q  \;\ll\, |\,\mathfrak a\,|^p \,+\, |\,\mathfrak b\,|^q \,+  \,\langle\, \alpha\,|\, \beta\,\rangle
 \end{equation}
 We arrive at the following point-wise estimate,
 \begin{equation}
 |\,\alpha\,|^p\;+\;|\,\beta \,|^q  \;\ll\, |\,\mathfrak a\,|^p \,+\, |\,\mathfrak b\,|^q \,+  \,\langle\, \alpha\,|\, \beta\,\rangle
 \end{equation}
 which, upon integration,  gives the desired estimate (\ref{basic estimate}).
\end{proof}

\section{Beyond the Natural Domain of Definition}

It makes sense to consider the equation (\ref{general equation}) in which the given pair $\, \mathfrak f = (\mathfrak a\,,\mathfrak b )\,$ lies in $\, \mathscr L^{\lambda p} (\mathbb X\,,\mathbb V)\,\times\,   \mathscr L^{\lambda q} (\mathbb X\,,\mathbb V)\,$,  so the proper space in which to seek the solution $ \,\phi =(\alpha\,,\, \beta )\,$ is $\,  \mathscr L^{\lambda p}_{_+} (\mathbb X\,,\mathbb V)\,\times \, \mathscr L^{\lambda q}_{_-} (\mathbb X\,,\mathbb V)\,$, where $\,\lambda\,$ is close enough to 1 so that $\,\lambda \,q \geqslant 1\,$ and  $\,\lambda \,p \geqslant 1\,$. Taking into an account the projection operators $\,\Pi_{_\pm} \,; \; \mathscr L^s(\mathbb X) \rightarrow \mathscr L^s_{_\pm}(\mathbb X) \,$, $\,\mathfrak q < s < \mathfrak p \,$, we shall restrict ourselves to the parameters $\, \lambda\,$ such that
\begin{equation}
  \mathfrak q \;<\,\, \lambda\, q \;\leqslant \;\lambda\, p \;< \;\mathfrak p
\end{equation}
The estimates are similar to those in the proof of Theorem \ref{existence and uniqueness}, with one principal ingredient. Let us begin with a point-wise inequality
$$|\,\mathfrak a +\,\alpha\,|^{\lambda p}\;+\;|\,\mathfrak b \,+\,\beta \,|^{\lambda q}  \;\ll\,  |\,\mathfrak a +\,\alpha\,|^{\lambda p } \;\ll\, \langle\;\mathfrak A(x, \mathfrak a + \alpha)\,\,|\,\,   (\mathfrak a + \alpha )\,|\,\mathfrak a + \alpha \,|^{(\lambda-1)p}\;\rangle \;$$
We decompose the vector field $\,\mathfrak v = (\mathfrak a + \alpha )\,|\,\mathfrak a + \alpha \,|^{(\lambda-1)p}\,$ as $\,\mathfrak v = \Pi_{_+}\mathfrak v \,+\, \Pi_{_-}\mathfrak v \;$, use the equation $\,\mathfrak A(x, \mathfrak a + \alpha) = \mathfrak b + \beta\,$, and integrate. By the orthogonality of $\,\beta \in \mathscr L^{\lambda q}_{_-}(\mathbb X)\,$ and  $\,\Pi_{_+}\mathfrak v \in \mathscr L^{\frac{\lambda q}{\lambda q - 1}}_{_+}(\mathbb X) \,$, we obtain

\[\begin{split}\int_\mathbb X |\,\mathfrak a +\,\alpha\,|^{\lambda p}\;+\;&|\,\mathfrak b \,+\,\beta \,|^{\lambda q}  \;\ll\, \int_\mathbb X \langle\,\mathfrak b\,\,|\,\, \Pi_+\mathfrak v\;\rangle\; + \int_\mathbb X \langle\,\mathfrak A(x, \mathfrak a + \alpha)\,\,|\,\,\Pi_-\mathfrak v\, \rangle \\& \ll \|\,\mathfrak b\,\|_{\lambda q } \,\|\,\Pi_+\mathfrak v\,\|_{\frac{\lambda q}{\lambda q - 1}}\;\;+\;\; \|\,(\mathfrak a +\alpha)^{p-1}\,\|_{\lambda q } \,\|\,\Pi_-\mathfrak v\,\|_{\frac{\lambda q}{\lambda q - 1}}\\&\ll  \;\|\,\mathfrak b\,\|_{\lambda q } \,\|\, \|\,\mathfrak a + \alpha \,\|_{\lambda p} ^{\lambda p - p + 1}\;\; +\;\; \|\,\mathfrak a +\alpha\,\|_{\lambda p }^{p-1} \,\|\,\Pi_-\mathfrak v\,\|_{\frac{\lambda q}{\lambda q - 1}}
\end{split}\]
where we have used boundedness of the operator $\,\Pi_+\,;\,\mathscr L^s (\mathbb X) \rightarrow \,\mathscr L^s (\mathbb X)\,$ with $\, s = \frac{\lambda q}{\lambda q - 1}\,$ and the identity $\, \|\mathfrak v \,\|_s \, = \|\,\mathfrak a + \alpha \,\|_{\lambda p} ^{\lambda p - p + 1}\,$.
With the aid of Young's inequality the term $\,\|\,\mathfrak a + \alpha \,\|_{\lambda p}\,$ can be absorbed by the left hand side. We thus obtain
\begin{equation}\label{1}
\int_\mathbb X |\,\mathfrak a +\,\alpha\,|^{\lambda p}\;+\;|\,\mathfrak b \,+\,\beta \,|^{\lambda q} \; \ll \; \int_\mathbb X \,|\mathfrak b\,|^{\lambda q } \; + \;\int_\mathbb X  \,|\,\Pi_-\mathfrak v\,\,|^s
\end{equation}
Now comes the principal ingredient in the proof; we have $\,\Pi_-\alpha = 0\,$, so
\begin{equation}\label{18}
 \Pi_-\mathfrak v \;= [\,\Pi_-(\mathfrak a + \alpha )^{1 + \epsilon}\,]  -  [\,\Pi_- (\mathfrak a + \alpha )\,]^{1+\epsilon}            +  (\Pi_ -\mathfrak a)^{1+\epsilon}
\end{equation}
where $\,\epsilon = (\lambda-1) p\,$.

The $\,s$-norm of $\,(\Pi_ -\mathfrak a)^{1+\epsilon}\,$ is controlled by the $\,\lambda p\;$- norm of $\,\mathfrak a\,$, because the operator  $\,\Pi_-\,$ is bounded in the space $\,\mathscr L^{(1+\epsilon) s}(\mathbb X) =  \mathscr L^{\lambda p}(\mathbb X)\,$; namely,
\begin{equation}\label{2}
 \;\int_\mathbb X  \,|\,\Pi_-\mathfrak v\,\,|^s \;\ll\; \int_\mathbb X \,|\mathfrak a|^{\lambda p}
\end{equation}
The first two terms in the right hand side of (\ref{18})  form a power type commutator. Precisely, we are dealing with a commutator of the linear operator $\,\Pi_{_-}\,$ and the nonlinear map $\,\mathfrak v \rightarrow |\mathfrak v|^\epsilon \mathfrak v \,$ with $\,\epsilon\,$ close to zero; the exponent $\,\epsilon\,$ can be both positive or negative. Here is what a complex interpolation method yields \cite{I5, IS1, IS2, RW, Sb}.

\begin{theorem} \label{commutator} Suppose a linear operator $\,\Pi \,;\, \mathscr L^r(\mathbb X, \mathbb V) \rightarrow  \,\mathscr L^r(\mathbb X, \mathbb V) \,$ is continuous for all exponents $\,r\,$ in the range $\,1\leqslant \mathfrak q \,< \, r\,<\,\mathfrak p \leqslant \infty\,$. Then for every  $\,1\leqslant \mathfrak q \,< \, s \,<\,\mathfrak p \leqslant \infty\,$, we have
$$
 \|\;\Pi (\mathfrak v ^{1+\epsilon}) \;  -  \,\,(\Pi \mathfrak v )^{1+\epsilon}\;\|_s\;      \; \ll \,|\epsilon\,|\cdot\|\, \mathfrak v^{1 + \epsilon}\,\|_s \,,\quad \textnormal{whenever} \;\;\;\; \frac{\mathfrak q}{s} < \, 1+ \epsilon \,< \frac{\mathfrak p}{s}$$
\end{theorem}

This gives:
$$
  \|\;[\,\Pi_-(\mathfrak a + \alpha )^{1 + \epsilon}\,]  -  [\,\Pi_- (\mathfrak a + \alpha )\,]^{1+\epsilon}\;\|_s\; \ll \,|\epsilon\,|\cdot\|\, (\mathfrak a + \alpha )^{1 + \epsilon}\,\|_s \,= \,|\epsilon| \cdot \|\, \mathfrak a + \alpha \,\|_{\lambda p}^{1+\epsilon}
$$

Finally, we chose $\,\epsilon\,$ small enough, meaning that $\,\lambda \approx 1\,$, to absorb this last term  by the left hand side of equation  (\ref{1}). Combining with  equation (\ref{2}), we obtain
\begin{equation}\label{3}
\int_\mathbb X |\,\mathfrak a +\,\alpha\,|^{\lambda p}\;+\;|\,\mathfrak b \,+\,\beta \,|^{\lambda q} \; \ll \; \int_\mathbb X \,|\mathfrak b\,|^{\lambda q } \; + \;\int_\mathbb X  |\,\mathfrak a\,|^{\lambda p}
\end{equation}
Let us introduce the notation
\begin{equation} \label{[]}
[\mathfrak f\, ] = |\mathfrak a |^p + |\mathfrak b |^q  \quad\quad\quad  [\phi\, ] = |\alpha|^p + |\beta |^q
\end{equation}

We just proved that:

\begin{theorem}
For all parameters $\,\lambda\,$ sufficiently close to 1; precisely, for
\begin{equation}\label{4}
\lambda \in (\lambda_- \,,\,\lambda _+ ) \;,\quad \textnormal{where }\;\;\;\;\frac{\mathfrak q}{q}\leqslant\lambda_ -\,< 1 \,<\lambda_+ \leqslant \frac{\mathfrak p}{p}
\end{equation}
 we have
\begin{equation}\label{3}
\int_\mathbb X [\phi]^\lambda  \; \ll \; \int_\mathbb X \,[\mathfrak f ]^\lambda
\end{equation}
\end{theorem}
\subsection{Acceptable solutions} We shall speak of   $\,(\lambda_- \,,\,\lambda _+ )\,$  as the \textit{fundamental interval} for the equation (\ref{general equation}), see Definition \ref{fundamental interval} for clarification  of this notion. The estimate (\ref{3}) raises an interesting question. Suppose we are given $\,\mathfrak f = (\mathfrak a\,,\mathfrak b )\,$ such that $\,\int_\mathbb X \,[\mathfrak f ]^\lambda\, < \infty\,$,  one might consider solutions $\,\phi = (\alpha\,,\beta)\,$ of the equation (\ref{general equation}) satisfying  $\,\int_\mathbb X [\phi\,]^\tau  < \infty\,$, for some $\, \tau \neq \lambda$. We assume that both exponents $\,\lambda\,$ and $\,\tau\,$ lie in the fundamental interval.  Is it true that  $\,\int_\mathbb X [\phi\,]^\lambda  < \infty\,$ and, therefore, the estimate (\ref{3}) holds? In general the answer to this question is "yes". However, the proof requires somewhat more elaborate variants of the power type commutators and the associated estimates, like in Theorem \ref{commutator}. The interested reader may wish to consult \cite{CGI} for such commutators to verify the above statement. This affirmative answer also settles the case $\, \lambda = 1\,$; that is, when the data $\,\mathfrak f = (\mathfrak a , \mathfrak b\,)\,$ belongs to the natural setting of the equation. The term \textit{acceptable solution} refers to a pair $\,\phi = (\alpha\,,\beta)\,$ satisfying Equation (\ref{general equation}) such that $\,\int_\mathbb X [\phi\,]^\tau  < \infty\,$, for some  $\,\tau \in (\lambda_-\,, \lambda_+)\,$.  Thus, within the fundamental interval for $\,\tau\,$, if the data  $\,[\,\mathfrak f\,]\,$ belongs to $\,\mathscr L^1(\mathbb X)\,$, then the acceptable solutions actually belong to the natural setting  of the equation. This means that $\,[\phi\,] \in \mathscr L^1(\mathbb X)\,$ and, in particular, such solutions are unique.

\subsection{A nonlinear counterpart of the Riesz Transforms}
Associated with each system (\ref{general equation}) is its operator
$$\,\mathfrak R \,;\;\mathscr L^p(\mathbb X)\times \mathscr L^q (\mathbb X) \rightarrow  \mathscr L^p(\mathbb X)\times \mathscr L^q (\mathbb X)\, ,\quad \textnormal{defined by}\;\;\;\;\mathfrak R \mathfrak f = \,\phi$$
We regard $\,\mathfrak R\,$ as a counterpart of the classical Riesz Transforms in $\,\mathscr L^2(\mathbb R^n, \mathbb V)\,$.
Actual extension of the domain of definition of $\,\mathfrak R\,$ may be accomplished based on estimates of the acceptable solutions to the system (\ref{general equation}). Such an approach is commonly realized by first estimating the operator in a conveniently chosen dense subspace of its natural domain of definition, and then extending it in accordance with the estimate.  We shall work with the following dense subspace

$$ \mathscr L^\ast(\mathbb X) \times \mathscr L^\ast(\mathbb X) \subset \mathscr L^p(\mathbb X) \times \mathscr L^q(\mathbb X) \;,\;\;\textnormal {where}\;\;\;\mathscr L^\ast(\mathbb X)\,=\, \bigcap_{1\leqslant s \,\leqslant \infty} \mathscr L^s(\mathbb X\,, \mathbb V) $$

\section{Marcinkiewicz Interpolation in a Nonlinear Context}
This section takes on the Marcinkiewicz interpolation theorem to the context of nonlinear equations (\ref{general equation}). The idea of decomposing and integrating functions over their level sets  is the core of the matter.  In our nonlinear setting, however,  one faces additional challenges because of insufficient additivity properties of the solutions to (\ref{general equation}); these properties proved very proficient in case of the linear operators.
\begin{definition}\label{weak type}
Whenever $\,\mathfrak q/q\;\leqslant\;\lambda\;\leqslant \;\mathfrak p/p\,$, the operator $\,\mathfrak R\,$ will be said to satisfy the \textit{weak $\lambda$-type} inequality  if
$$
 \textnormal{meas}\{\,x\,;\, [\,\mathfrak R\mathfrak f (x)\,] \; >\; t\;\} \; \ll \; \frac{1}{t^\lambda}\,\int _\mathbb X [\,\mathfrak f\, ] ^\lambda \;, \quad \textnormal{for every}\;\;\ t> 0
$$
Here the implied constant is independent of $\,\mathfrak f =(\mathfrak a , \mathfrak b)\, \in  \mathscr L^\ast(\mathbb X) \times \mathscr L^\ast(\mathbb X)\,$. Thus, in particular, the induced solution $\,\mathfrak R\mathfrak f = (\alpha,\beta) \in  \mathscr L^p(\mathbb X) \times \mathscr L^q(\mathbb X)\,$ is implicitly assumed to  belong to $\,\mathscr L^{\lambda p}(\mathbb X) \times \mathscr L^{\lambda q}(\mathbb X)\,$ as well.
\end{definition}
\subsection{The main result}
The following generalization of Marcinkiewicz Theorem turns out particularly useful when applied to the $\,p$-harmonic type PDEs..
\begin{theorem}\label{Marcinkiewicz}
Let $\,\lambda_- \,$ and $\,\lambda _+\,$ be exponents, $\, \mathfrak q/q \,\leqslant\lambda_ -\,< 1 \,<\lambda_+ \leqslant \mathfrak p/p\,$, for which  $\,\mathfrak R\,$ is both of weak $\,\lambda_-$\,-type and weak $\,\lambda_+$ \,-type. Then for every $\,\tau \in (\lambda_- \,,\,\lambda _+ )\,$  the operator $\mathfrak R\,$ is of strong $\tau$-type, meaning that
\begin{equation}\label{5}
\int_\mathbb X [\,\mathfrak R \mathfrak f\,]^\tau  \; \ll \; \int_\mathbb X \,[\mathfrak f\, ]^\tau \;\;\;\;\textnormal{ for all }\;\;\;\; \mathfrak f \in \mathscr L^\ast(\mathbb X) \times \mathscr L^\ast(\mathbb X).
\end{equation}
\end{theorem}
At this stage  we shall make the following

\begin{definition}[Fundamental Interval] \label{fundamental interval} The largest interval $\, (\lambda_- \,,\,\lambda _+ )\,$, $\;\mathfrak q/q\leqslant\lambda_ -< 1 <\lambda_+ \leqslant \mathfrak p/p\,$, for which  $\,\mathfrak R\,$ acts from $\mathscr L^{\lambda_-}(\mathbb X) \into  \mathscr L_{\textnormal{weak}}^{\lambda_-}(\mathbb X)\,$  and from $\mathscr L^{\lambda_+}(\mathbb X) \into  \mathscr L_{\textnormal{weak}}^{\lambda_+}(\mathbb X)\,$, will be called the \textit{fundamental interval} of the equation (\ref{general equation}).
\end{definition}
\begin{remark}
The point to make here is that in Section 4 we were able to establish the estimate (\ref{3}) only for $\,\lambda's\,$ sufficiently close to 1. In the above  definition, however,  we do not insist on this assumption. Thus Theorem \ref{Marcinkiewicz}\, broaden's estimate (\ref{3}) to be true in the entire fundamental interval.
\end{remark}

\begin{proof}
Let us refresh the equations
$$ \mathfrak A(x, \mathfrak a +\alpha)  =  \mathfrak b + \beta \quad\quad\quad \mathfrak B(x, \mathfrak b +\beta)  =  \mathfrak a + \alpha $$
and remind our notation
$$ \mathfrak{f} =(\mathfrak a ,\mathfrak b )  \quad\quad\quad \phi = ( \alpha , \beta ) = \mathfrak R \mathfrak f \quad\quad \int_\mathbb X \langle \,\alpha(x)\, |\, \beta(x) \,\rangle \;\textrm d x  =  0 $$
For further notational convenience we introduce
$$  \mathcal A =  \mathfrak a +\alpha\,,  \quad    \mathcal B =  \mathfrak b +\beta \;,\;\quad \mathcal H = (\mathcal A,\mathcal B) = \mathfrak f + \mathfrak R \mathfrak f \;, \quad\quad [\,\mathcal H\,] = |\,\mathcal A\,| ^p + \,|\,\mathcal B\,|^q $$
Thus the system of equations abbreviates to:
\begin{equation}\label{new setting}
\mathfrak A(x, \mathcal A)  =  \mathcal B \quad\quad\textnormal{or, equivalently }\quad \quad \mathfrak B(x, \mathcal B)  =  \mathcal A
\end{equation}
and we are reduced to showing that
\begin{equation}
\int_\mathbb X [\,\mathcal H(x)\,]^\tau \,\textnormal d x \;\ll \; \int_\mathbb X [\,\mathfrak f(x)\,]^\tau \,\textnormal d x\;, \quad \textnormal{ for}\;\; \tau \in (\lambda_-\,,\,\lambda_+\,)
\end{equation}
which is certainly true for $\,\tau = 1\,$, so we need only consider $\,\tau\,$ between  1 and $\,\lambda = \lambda _\pm\,$;  precisely, for $\,\tau\,$ satisfying
$$
  0 \leqslant \frac{\tau - 1}{\lambda - 1} \,< \, 1\,,  \quad\,\textnormal{where either}\;\;\lambda = \lambda_- < 1\;\;\;\;\textnormal{or}\;\;\;\;\lambda  = \lambda_ + > 1
$$
We shall demonstrate the proof for $\, 1 < \tau < \lambda = \lambda _+\,$. The case $\,\lambda_- =\lambda < \tau < 1\, $ goes in an exactly similar way, which will be emphasized several times as the proof develops.
We observe that whenever  vector fields $\, \mathcal A\,$ and $ \,\mathcal B \,$ are coupled by the relations \,(\ref{new setting}), they are comparable in the following fashion
$$ |\mathcal A|^p \; \approxeq \langle \,\mathcal A \,|\, \mathfrak A(x, \mathcal A)\,\rangle =\;\langle \,\mathcal A\,| \,\mathcal B\,\rangle \;=\; \langle \,\mathfrak B(x, \mathcal B)\,| \,\mathcal B\,\rangle  \approxeq |\mathcal B|^q$$
\subsection{Marcinkiewicz decomposition}
For $ t\geqslant 0 $ and a given pair $\mathfrak f (x) = (\mathfrak a ,\mathfrak b ) $ we consider a decomposition:
$$  \mathfrak f (x)  =  \mathfrak f ^t(x) + \mathfrak f _t(x) \; $$
where
\[ \mathfrak f ^t(x)  = (\mathfrak a ^t ,\mathfrak b ^ t ) =
\begin{cases} (\mathfrak a ,\mathfrak b ) =\mathfrak f (x) \;, \quad &\textnormal{if} \;\;\; [\,\mathfrak f (x)\,] = |\mathfrak a(x) |^p + |\mathfrak b(x) |^q  > t\\
 0 \quad &\textnormal{if} \;\;\; [\,\mathfrak f (x)\,] = |\mathfrak a(x) |^p + |\mathfrak b(x)  |^q  \leqslant t
\end{cases}
  \]
Similarly
\[ \mathfrak f _t(x)  = (\mathfrak a _t ,\mathfrak b _t ) =
\begin{cases} (\mathfrak a ,\mathfrak b ) =\mathfrak f (x) \;, \quad &\textnormal{if} \;\;\; [\mathfrak f (x) \,] \leqslant t\\
 0 \quad &\textnormal{if} \;\;\; [\mathfrak f (x) \,] > t
\end{cases}
  \]
Then we solve the equations for $ \phi ^{\,t}(x)  = (\alpha ^t ,\beta ^ t ) \in \mathscr L^p_+(\mathbb X)\times \mathscr L^q_-(\mathbb X)$ and $ \phi _{\,t}(x)  = (\alpha _t\, ,\beta _ t ) \in \mathscr L^p_+(\mathbb X)\times \mathscr L^q_-(\mathbb X)$, respectively.
  $$ \mathfrak A(x, \mathfrak a^t +\alpha^t)  =  \mathfrak b^t + \beta^t  \quad\quad\quad  \mathfrak A(x, \mathfrak a_t +\alpha_t)  =  \mathfrak b_t + \beta_t  $$
where
$$ \int _\mathbb X \langle \,\alpha ^t(x)\,|\, \beta ^t(x)\,\rangle \;\textrm d x  = 0  \quad\quad\quad  \int _\mathbb X \langle \,\alpha _t(x)\,|\, \beta _t(x)\,\rangle \;\textrm d x  = 0$$
\textit{Caution.} It is not true in general that  $\,\phi(x) = \phi ^{\,t}(x) + \phi _{\,t}(x)\, $ .

Let us denote
$$  \mathcal A^t =  \mathfrak a^t +\alpha^t \,, \quad     \mathcal B^t =  \mathfrak b^t +\beta^t \quad\textrm{and} \quad   \mathcal A_t =  \mathfrak a_t +\alpha_t\,,  \quad     \mathcal B_t =  \mathfrak b_t +\beta_t$$
$$ \mathcal H_t = (\mathcal A_t,\mathcal B_t) \quad\quad\quad [\,\mathcal H_t\,] = |\,\mathcal A_t\,| ^p + \,|\,\mathcal B_t\,|^q $$
$$ \mathcal H^t = (\mathcal A^t,\mathcal B^t) \quad\quad\quad [\,\mathcal H^t\,] = |\,\mathcal A^t\,| ^p + \,|\,\mathcal B^t\,|^q $$
We shall also introduce the \textit{energy integrands}

$$  E^t(x) = \langle\, \mathcal A \,-\mathcal A^t \,| \, \mathcal B \,-\mathcal B^t\rangle\;\geqslant   0 \,, \quad  E_t(x) = \langle\, \mathcal A \,-\mathcal A_t \,| \, \mathcal B \,-\mathcal B_t\rangle\;\geqslant   0  $$

\begin{lemma}[point-wise inequalities]\label{19}
We have
\begin{equation}\label{10}
[\mathcal H  -\mathcal H_t ]  \; \ll\; E_t  \; + [\mathcal H ]\;\;\;\;\;\;\textnormal{and}\;\;\;\;\,[\mathcal H  -\mathcal H^t ]  \; \ll\; E^t  \; + [\mathcal H\,]
\end{equation}
\begin{equation}\label{11}
[\mathcal H \,-\,\mathcal H_t\,]  \; \ll\; E_t  \; + [\mathcal H_t\, ]\;\;\;\;\;\;\textnormal{and}\;\;\;\;\,[\mathcal H\,-\,\mathcal H^t ]  \; \ll\; E^t  \; + [\mathcal H^t\,]
\end{equation}

\end{lemma}
\begin{proof} For all exponents $\, 1 < p < \infty\, $, we can write

$$
 | \mathcal A  - \mathcal A_t |^p \ll |\,\mathcal A - \mathcal A_t\,|^2 (\,|\mathcal A| + |\mathcal A_t|)^{p-2} + \begin{cases} \textnormal{either}\;\;\;|\mathcal A\,|^ p\\
 \;\textnormal{or}\;\;\;|\mathcal A_t\,|^ p
\end{cases}
$$
$$
\ll \langle\, \mathcal A \,-\mathcal A_t \,\,| \,\, \mathfrak A(x,\mathcal A) \,-\mathfrak A(x,\mathcal A_t\,)\rangle\,\;\; + \begin{cases} \textnormal{either}\;\;\;[\,\mathcal H \,]\\
 \;\textnormal{or}\;\;\;[\,\mathcal H_t\,]
\end{cases}
$$
$$
\ll \langle\, \mathcal A \,-\mathcal A_t \,\,| \,\, \mathcal B \,-\,\mathcal B_t\,\rangle\,\;\; + \begin{cases} \textnormal{either}\;\;\;[\,\mathcal H \,]\\
 \;\textnormal{or}\;\;\;[\,\mathcal H_t\,]
\end{cases}
\;\;=\;\; E_t \;+\; \begin{cases} \textnormal{either}\;\;\;[\,\mathcal H \,]\\
 \;\textnormal{or}\;\;\;[\,\mathcal H_t\,]
 \end{cases}
$$

Similarly, for the exponent $\, 1 < q < \infty\,$,

$$
 | \mathcal B  - \mathcal B_t |^q \ll |\,\mathcal B - \mathcal B_t\,|^2 (\,|\mathcal B| + |\mathcal B_t|)^{q-2} + \begin{cases} \textnormal{either}\;\;\;|\mathcal B\,|^ q\\
 \;\textnormal{or}\;\;\;|\mathcal B_t\,|^ q
\end{cases}
$$
$$
\ll \langle\, \mathcal B \,-\mathcal B_t \,\,| \,\, \mathfrak B(x,\mathcal B) \,-\mathfrak B(x,\mathcal B_t\,)\rangle\,\;\; + \begin{cases} \textnormal{either}\;\;\;[\,\mathcal H \,]\\
 \;\textnormal{or}\;\;\;[\,\mathcal H_t\,]
\end{cases}
$$
$$
\ll \langle\, \mathcal B \,-\mathcal B_t \,\,| \,\, \mathcal A \,-\,\mathcal A_t\,\rangle\,\;\; + \begin{cases} \textnormal{either}\;\;\;[\,\mathcal H \,]\\
 \;\textnormal{or}\;\;\;[\,\mathcal H_t\,]
\end{cases}
\;\;=\;\; E_t \;+\; \begin{cases} \textnormal{either}\;\;\;[\,\mathcal H \,]\\
 \;\textnormal{or}\;\;\;[\,\mathcal H_t\,]
 \end{cases}
$$

Adding up the above inequalities we conclude with the desired estimate corresponding to the lower subscript $\, t > 0\,$.

$$
[\mathcal H  -\mathcal H_t ] = | \mathcal A  - \mathcal A_t |^p \;+\; | \mathcal B  - \mathcal B_t |^q \;\ll\quad\; E_t \;+\; \begin{cases} \textnormal{either}\;\;\;[\,\mathcal H \,]\\
 \;\textnormal{or}\;\;\;[\,\mathcal H_t\,]
 \end{cases}
$$
In an exactly the same way we derive the inequalities for the  upper superscript $t > 0\,$.
\end{proof}
Now, the $(p,q)$-orthogonality comes into play when integrating the truncated energy functions $\, E^t\, $ and $\, E_t\,$,
\begin{lemma}[The energy estimates] \label{energy} We have,
$$
 \mathscr E^t \mathfrak f := \int_\mathbb X E^t(x) \; \textnormal d x \; \ll \;\int_\mathbb X [\,\mathfrak f_t\,] +  \int_\mathbb X [\,\mathfrak f_t\,]^{\frac{1}{p}} \,[\,\mathcal H \,]^{\frac{1}{q}}\; +\; \int_\mathbb X [\,\mathfrak f_t\,]^{\frac{1}{q}} \,[\,\mathcal H \,]^{\frac{1}{p}}
$$
Similarly,
$$
 \mathscr E _t\mathfrak f  := \int_\mathbb X E_t(x) \; \textnormal d x  \;\ll \;\int_\mathbb X [\,\mathfrak f^t\,] +  \int_\mathbb X [\,\mathfrak f^t\,]^{\frac{1}{p}} \,[\,\mathcal H \,]^{\frac{1}{q}}\;+\;  \int_\mathbb X [\,\mathfrak f^t\,]^{\frac{1}{q}} \,[\,\mathcal H \,]^{\frac{1}{p}}
$$
\end{lemma}
\begin{proof} Since $\, \mathfrak a - \mathfrak a^t = \mathfrak a _t \,$  and  $\, \mathfrak b - \mathfrak b ^t = \mathfrak b _t \,$, we can write
\[\begin{split} E^t &= \langle\, \mathcal A \,-\mathcal A^t \,\,| \,\, \mathcal B \,-\mathcal B^t\,\rangle\; = \langle\,\, \mathfrak a_t + \alpha - \alpha ^t  \,\,\,| \,\,\, \mathfrak b_t  +  \beta - \beta^t \,\rangle\;=  \\& \;  - \, \langle \,\mathfrak a_t \,\,| \,\,\mathfrak b_t\,\rangle \;+\; \langle\,\mathfrak a_t \,\, | \,\,\mathcal B \,-\mathcal B^t\,  \rangle +    \langle\,\mathfrak b_t \,\, | \,\,\mathcal A \,-\mathcal A^t\,\rangle   +
\langle\,\alpha - \alpha ^t \,\, |\, \,\beta - \beta^t \;\rangle\;\;\ll\; \\& \; [\,\mathfrak f_t\,]\;+\;\;[\,\mathfrak f_t\,]^{\frac{1}{p}} ( \,E^t\,+ [\,\mathcal H\,]\,) ^{\frac{1}{q}}\;+\; [\,\mathfrak f_t\,]^{\frac{1}{q}} (\,E^t\,+\,[\,\mathcal H\,]\,) ^{\frac{1}{p}}\;+\; \langle\,\alpha - \alpha ^t \, |\, \,\beta - \beta^t \,\rangle\;
\end{split}\]
 In the last step we have used inequalities in Lemma \ref{19}.
Then, with the aid of Young's inequality, the term $\,E^t\,$ can be absorbed by the left hand side.
$$
E^t \;\ll\; [\,\mathfrak f_t\,]\;+\;\;[\,\mathfrak f_t\,]^{\frac{1}{p}} ( \, [\,\mathcal H\,]\,) ^{\frac{1}{q}}\;+\; [\,\mathfrak f_t\,]^{\frac{1}{q}} (\,[\,\mathcal H\,]\,) ^{\frac{1}{p}}\;+\; \langle\,\alpha - \alpha ^t \, |\, \,\beta - \beta^t \,\rangle\;
$$
\vskip0.3cm
Since $ \alpha - \alpha ^t \in \mathscr L^p_+(\mathbb X)$ \; and \; $ \beta - \beta ^t \in \mathscr L^q_-(\mathbb X)\,,$ the integral of the last term vanishes. Hence, integrating over $\,\mathbb X\,$ yields the first inequality of the lemma. The second one is proven in much the same way.
\end{proof}
Now  the assumption that the operator $\,\mathfrak R\,$ is $\mathscr L^\lambda_{\textrm{weak}}$-type comes into play.
\begin{lemma}

Let the exponent $\tau$ lie between 1 and $\lambda$; that is, $ 0 < \frac{\tau - 1}{\lambda - 1}<1 $. Then
\begin{equation}\label{-}
\int_0^\infty  t^{\tau-1} \textnormal{meas}\{ x ;\; [\,\mathcal H_t(x) \,] \,> t \}\,\textnormal d t \;\ll \int_\mathbb X [\mathfrak f\,]^\tau \quad \textnormal{if}\;\; 1 < \tau < \lambda = \lambda_+
\end{equation}
and
\begin{equation}\label{^}
\int_0^\infty  t^{\tau-1} \textnormal{meas}\{ x ; \;[\,\mathcal H^t(x)\, ] \,> t \}\,\textnormal d t \;\ll \int_\mathbb X [\mathfrak f\,]^\tau \quad \textnormal{if}\;\; \lambda_- = \lambda < \tau < 1
\end{equation}
\end{lemma}

\begin{proof}
We shall be concerned with the level sets  $\,\{\,x\,;\; \Gamma(x) > c\,t\,\}\,$, where $\,c\,$ will be a constant, again called implied constant, and the parameter $\,t\,$ will run from $\,0\,$ to $\,\infty\,$. This implied constant may alter from line to line, but this will have no effect on the subsequent estimates of the integrals over the entire domain $\,\mathbb X\,$.
Therefore, whenever it is convenient, we shall abbreviate the notation  $\,\{\,x\,;\; \Gamma(x) > c\,t\,\}\,$ to $\,\{\,x\,;\; \Gamma(x) \succ t\,\}\,$.\\ Let us take the case $\,1 < \tau < \lambda = \lambda_+\,$. We begin with the point-wise inequality $\,[\,\mathcal H_t(x)\, ] \,= \, [\,\mathfrak f_t \,+\mathfrak R \mathfrak f_t\,] \,\ll [\,\mathfrak f_t \,] \,+\, [\,\mathfrak R \mathfrak f_t\,]\,$. Hence

$$
 \textnormal{meas}\{ x ; \;[\,\mathcal H_t(x)\, ] \,> t \} \,\ll\, \textnormal{meas}\{ x ; \;[\,\mathfrak f_t(x)\, ] \,\succ t \} \,+\; \textnormal{meas}\{ x ; \;[\,\mathfrak R \mathfrak f_t(x)\, ] \,\succ t \}
$$
Since the identity operator and $\,\mathfrak R\,$ are both of  $\mathscr L^\lambda_{\textrm{weak}}$-type, we can write
\[\begin{split}
&\int_0^\infty  t^{\tau-1} \textnormal{meas}\{ x ;\; [\,\mathcal H_t(x) \,] \,> t \}\,\textnormal d t \;\ll\\&
\int_0^\infty  t^{\tau-1} \textnormal{meas}\{ x ;\; [\,\mathfrak f_t(x) \,] \,\succ t \}\,\textnormal d t \; + \; \int_0^\infty  t^{\tau-1} \textnormal{meas}\{ x ;\; [\,\mathfrak R\mathfrak f_t(x) \,] \,\succ t \}\,\textnormal d t\,\\&\;\ll
\int_0^\infty  t^{\tau-1} \Big( t^{-\lambda} \int_\mathbb X [\,\mathfrak f_t(x)\,]^\lambda \Big)\,\textnormal d t \;\; + \;\; \int_0^\infty  t^{\tau-1} \Big( t^{-\lambda} \int_\mathbb X [\,\mathfrak f_t(x)\,]^\lambda \Big)\,\textnormal d t \;\\&
=\;2 \;\int_0^\infty  t^{\tau-\lambda -1} \Big( \int_{[\,\mathfrak f\,] \leqslant t } [\,\mathfrak f(x)\,]^\lambda \,\textnormal d x\,\Big)\,\textnormal d t\;\, =\;2\,\int_\mathbb X [\,\mathfrak f(x)\,]^\lambda \Big( \int_{[\,\mathfrak f\,]}^\infty t^{\tau - \lambda - 1 } \,\textnormal d t  \Big)\, \textnormal d x\;\\&\;=\; \frac{2}{\lambda - \tau}\int_\mathbb X [\,\mathfrak f(x) \,]^ \tau \textnormal d x\;\ll\; \int_\mathbb X [\,\mathfrak f(x) \,]^ \tau \textnormal d x\
\end{split}\]
as desired. \\The case $\, \lambda_- = \lambda < \tau < 1\,$ is treated in much the same way; it begins  with the point-wise inequality $\,[\,\mathcal H^t(x)\, ] \,= \, [\,\mathfrak f^t \,+\mathfrak R \mathfrak f^t\,] \,\ll [\,\mathfrak f^t \,] \,+\, [\,\mathfrak R \mathfrak f^t\,]\,$. We leave the details to the reader.
\end{proof}
\section{The Interpolation Estimate, proof of Theorem\; \ref{Marcinkiewicz}}
We aim to show that
\begin{equation}
\int_\mathbb X [\,\phi(x)\,]^\tau \,\textnormal d x \;\ll \; \int_\mathbb X [\,\mathfrak f(x)\,]^\tau \,\textnormal d x
\end{equation}
Equivalently,
\begin{equation}
 \int_\mathbb X [\,\mathcal H(x)\,]^\tau \,\textnormal d x \;\;=\;\;\int_\mathbb X [\,\mathfrak f(x)  + \phi(x)\,]^\tau \,\textnormal d x \;\ll \; \int_\mathbb X [\,\mathfrak f(x)\,]^\tau \,\textnormal d x
\end{equation}
Let us discuss in details the case  $ 1< \tau < \lambda $. We make use of (\ref{11}), which yields  $\, [\mathcal H \,]  \; \ll\; E_t  \; + [\mathcal H_t\, ]\,$, and the energy estimate  in Lemma \ref{energy}\,, to obtain

\[\begin{split}
 &\int_\mathbb X \,[\,\mathcal H(x)\,] ^\tau\,\textnormal d x   =  \tau \int_0^\infty t^{\tau -1 } \textnormal{meas} \{ x \;; \; [\,\mathcal H(x)\,]  > t\;\} \;\textnormal d t  \\ & \quad\quad\quad \quad\quad\ll \tau \int_0^\infty t^{\tau -1 } \textnormal{meas} \{ x \;; \; [\,\mathcal H_t(x)\,]  \succ t\;\} \;\textnormal d t
\\& \quad\quad\quad \quad\quad+ \tau \int_0^\infty t^{\tau -1 } \textnormal{meas} \{ x \;; \; [\,E_t(x)\,]  \succ t\;\} \;\textnormal d t  \\ &
\quad\quad\quad \quad\quad\ll \int_\mathbb X [\mathfrak f\,]^\tau   \, +  \int_0^\infty t^{\tau -2 }\mathscr E_t\mathfrak f   \;\textnormal d t   \\& \ll \int_\mathbb X [\mathfrak f\,]^\tau  \; +
 \,  \int_0^\infty t^{\tau -2 }   \Big (   \int_\mathbb X [\,\mathfrak f^t\,] +  \int_\mathbb X [\,\mathfrak f^t\,]^{\frac{1}{p}} \,[\,\mathcal H \,]^{\frac{1}{q}} \;+\; \int_\mathbb X [\,\mathfrak f^t\,]^{\frac{1}{q}} \,[\,\mathcal H \,]^{\frac{1}{p}}\Big ) \,\textnormal d t
 \; \\& \ll\;\int_\mathbb X [\mathfrak f\,]^\tau  \; +
 \,    \int_0^\infty t^{\tau -2 }  \Big(\int_{[\mathfrak f ] > t } [\,\mathfrak f\,]\;+\;[\,\mathfrak f\,]^{\frac{1}{p}} \,[\,\mathcal H \,]^{\frac{1}{q}} \,+\, [\,\mathfrak f\,]^{\frac{1}{q}} \,[\,\mathcal H \,]^{\frac{1}{p}}\, \Big ) \,\textnormal d t \\& =
 \frac{\tau}{\tau -1} \int_\mathbb X [\mathfrak f\,]^\tau  \; + \frac{1}{\tau -1} \int_\mathbb X \Big( [\,\mathfrak f\,]^{\tau - \frac{1}{q}} \,[\,\mathcal H \,]^{\frac{1}{q}}\,+\, [\,\mathfrak f\,]^{\tau - \frac{1}{p}} \,[\,\mathcal H \,]^{\frac{1}{p}}\Big)
\end{split}
\]
The last equality is just an application of Fubini's Theorem. It is at this stage that we may (and will do)  separate $\,[\,\mathcal H\,]\,$ from $\,[\,\mathfrak f\,]\,$ without damaging the subsequent estimates. By H\"{o}lder's inequality it follows that,
\[\begin{split}
 \int_\mathbb X &\,[\,\mathcal H(x)\,] ^\tau\,\textnormal d x \; \ll \;\\&
  \int_\mathbb X [\mathfrak f\,]^\tau  \; +  \Big(\int_\mathbb X [\,\mathfrak f\,]^\tau \Big )^{1 - \frac{1}{\tau q}} \Big( \int _\mathbb X \,[\,\mathcal H \,]^\tau \Big) ^{\frac{1}{\tau q}} \; +  \Big(\int_\mathbb X [\,\mathfrak f\,]^\tau \Big )^{1 - \frac{1}{\tau p}} \Big( \int _\mathbb X \,[\,\mathcal H \,]^\tau \Big) ^{\frac{1}{\tau p}}
\end{split}\]

Finally, with the aid of Young's inequality the term $\,\int _\mathbb X \,[\,\mathcal H \,]^\tau \,$ in the right hand side can be absorbed by the left hand side. It  results in the desired estimate
 $$ \int_\mathbb X \,[\,\mathfrak R\mathfrak f(x)\,] ^\tau\,\textnormal d x \ll \int_\mathbb X \,[\,\mathfrak f(x)\,] ^\tau\,\textnormal d x \;+ \int_\mathbb X \,[\,\mathcal H(x)\,] ^\tau\,\textnormal d x \; \ll \; \int_\mathbb X \,[\,\mathfrak f(x)\,] ^\tau\,\textnormal d x \;$$

We leave it to the reader to verify, in an exactly similar fashion, the case $ \,\lambda < \tau < 1\,$; simply
the subscript $\,t\,$ should be replaced by superscript $\,t\,$.

\end{proof}
Finally, taking $\mathfrak b = 0\,$ in the above estimates  the proof of Theorem \ref{0} goes through with hardly any changes.

\bibliographystyle{amsplain}

\end{document}